\documentclass[11pt,a4paper]{article}

\usepackage[margin=1.3in]{geometry}
\usepackage[utf8]{inputenc}
\usepackage[T1]{fontenc}
\usepackage{enumerate}
\usepackage{pgfplots}
\pgfplotsset{compat=1.17}
\usepackage{booktabs}
\usepackage{tablefootnote}
\usepackage{float}
\usepackage{subcaption}
\usepackage{enumitem}
\usepackage[colorlinks, pagebackref=true]{hyperref}
\hypersetup{
  linkcolor=[rgb]{0.3,0.3,0.6},
  citecolor=[rgb]{0.2, 0.6, 0.2},
  urlcolor=[rgb]{0.6, 0.2, 0.2}
}
\usepackage{standalone}
\usepackage{mathtools, amsthm, amsfonts, amssymb, commath}
\usepackage{tikz}
\usetikzlibrary{decorations.markings}
\usetikzlibrary {arrows.meta}

\usepackage{authblk}

\usepackage{listings}
\usepackage{accents} %
\usepackage{microtype}
\usepackage{xcolor}

\PassOptionsToPackage{usenames,dvipsnames}{xcolor}
\usetikzlibrary {arrows.meta}
\usetikzlibrary {patterns}
\usepgfplotslibrary{fillbetween}

\usepackage{thmtools, thm-restate} 
\declaretheorem[name=Theorem, parent=section]{theorem}
\declaretheorem[name=Corollary, sibling=theorem]{corollary}

\declaretheorem[name=Lemma, sibling=theorem]{lemma}
\declaretheorem[name=Lemma, numbered=no]{lemma*}

\theoremstyle{definition}
\declaretheorem[name=Definition, sibling=theorem]{definition}
\declaretheorem[name=Remark, sibling=theorem]{remark}
\declaretheorem[name=Problem, sibling=theorem]{problem}

\theoremstyle{remark}

\newcommand{\RR}{\mathbb{R}}
\newcommand{\QQ}{\mathbb{Q}}
\newcommand{\NN}{\mathbb{N}}
\newcommand{\ZZ}{\mathbb{Z}}

\let\eps\varepsilon

\newcommand{\open}{\mathrm{o}}

\DeclareMathAccent{\wtilde}{\mathord}{largesymbols}{"65}

\newcommand{\alphag}{\alpha_{\mathrm{grp}}}
\newcommand{\Thetag}{\Theta_{\mathrm{grp}}}

\definecolor{DarkGreen}{cmyk}{1,0,0.6,0.2}
\definecolor{DarkGreen2}{RGB}{46,148,0}

\newcommand{\vect}[1]{#1}

\title{\vspace{-2.5em}A group-theoretic approach to Shannon capacity of graphs and a limit theorem from lattice packings}

\author[1]{Pjotr Buys}
\author[2]{Sven Polak}
\author[1]{Jeroen Zuiddam}

\affil[1]{University of Amsterdam}
\affil[2]{Tilburg University}

\date{\today}

\begin{document}
\maketitle
\small
\centerline{\textbf{Abstract}}
\vspace{0.5em}
We develop a group-theoretic approach to the Shannon capacity problem. Using this approach we extend and recover, in a structured and unified manner, various families of previously known lower bounds on the Shannon capacity.
Bohman (2003) proved that, in the limit $p\to\infty$, the Shannon capacity of cycle graphs~$\Theta(C_p)$ converges to the fractional clique covering number, that is, $\lim_{p \to \infty} p/2 - \Theta(C_p) = 0.$ We strengthen this result by proving that the same is true for all fraction graphs: $\lim_{p/q \to \infty} p/q - \Theta(E_{p/q}) = 0.$ Here the fraction graph~$E_{p/q}$ is the graph with vertex set $\ZZ/p\ZZ$ in which two distinct vertices are adjacent if and only if their distance mod $p$ is strictly less than~$q$. 
We obtain the limit via the group-theoretic approach. In particular, the independent sets we construct in powers of fraction graphs are subgroups (and, in fact, lattices). %
Our approach circumvents known barriers for structured (``linear'') constructions of independent sets of Calderbank--Frankl--Graham--Li--Shepp (1993) and Guruswami--Riazanov (2021).

\normalsize
\section{Introduction}
\label{sec:intro}

Determining the Shannon capacity of graphs, introduced by Shannon in~1956 to model zero-error communication over noisy channels \cite{MR0089131}, is a long-standing open problem in discrete mathematics, information theory and combinatorial optimization \cite{MR1658803, MR1919567, MR1956924}.
In graph-theoretic terms, this parameter measures the rate of growth of the largest independent set in strong powers of a graph. 

The Shannon capacity has been studied from many angles, which led to a variety of upper bound methods (e.g., the Lovász theta function \cite{Lovasz79} and (fractional) Haemers bound \cite{haemers1979some, Bukh18}), lower bound constructions---which have been mostly ad hoc---(e.g., \cite{MR0337668, bohman2005limit, MR1967195, MR3906144, romera-paredes_mathematical_2024}), and structural results \cite{alon1998shannon, MR2234473, MR4039606, MR4357434,  wigderson2022asymptotic}. 
Determining the Shannon capacity is notoriously hard; the Shannon capacity is not known for odd cycles of length $\geq7$, for instance.

Inspired by recent ``orbit'' constructions for (the best-known) lower bounds on the Shannon capacity of small odd cycles \cite{MR3906144, deboer2024asymptotic}, we develop a \emph{group-theoretic approach} to Shannon capacity (which for the class of graphs that we focus on is about constructing lattices with good properties). The main idea of this approach is to construct independent sets in Cayley graphs that are subgroups. Indeed, Shannon's~\cite{MR0089131} famous independent set in the square of the five-cycle has such a form: 
\[
\{ t \cdot (1,2) : t \in \ZZ_5\} \subseteq \ZZ_5^2.
\]
How can we find such constructions? Can they give good lower bounds on Shannon capacity? We make progress on both questions.

We focus on a family of Cayley graphs that  naturally generalizes odd cycle graphs, namely Cayley graphs for groups $\ZZ_p = \ZZ/p\ZZ$ and generating sets $\{\pm 1,\ldots, \pm (q-1)\}$.
Such graphs, which we refer to as fraction graphs, have indeed played a central role in aforementioned orbit constructions for small odd cycles, in particular the seven cycle and fifteen cycle. In fact, it will also be natural to work with the infinite version of these graphs, where the group is the circle group, which we refer to as circle graphs.

Subgroup independent sets in fraction graphs and circle graphs correspond, as we will see, precisely to lattices with good size and distance properties, giving rise to a \emph{lattice approach} to lower bounding the Shannon capacity of such graphs, reminiscent of optimal sphere packing constructions \cite{MR3664816, MR3664817, campos2023newlowerboundsphere}. The Shannon capacity problem on fraction graphs can indeed be rephrased as the problem of packing cubes with specified size in a high-dimensional unit-length torus \cite{MR0337668, MR1401751}.

Central directions in this paper are: (1) methods for proving that a given lattice is good for the lattice approach, (2) a construction of a family of good lattices and as a result many new lower bounds on the Shannon capacity, and (3) a proof that these lattice constructions are optimal in the limit of large fraction graphs.

We summarize our main results here and discuss them in more detail in the rest of the introduction:
\begin{itemize}
    \item \textbf{Group-theoretic approach.} We develop a group-theoretic approach to the Shannon capacity problem, in particular for fraction graphs and circle graphs, in which case it requires   
    constructing lattices with good size and distance properties. 
    While similar in spirit to the linear Shannon capacity of Guruswami and Riazanov~\cite{LinearShannonCapacity}, our approach avoids related barriers (by identifying equivalent fraction graphs) \cite{MR1210400, LinearShannonCapacity}. %
    \item \textbf{New Shannon capacity lower bounds via lattices.} We construct an explicit family of lattices suitable for the above approach and thus obtain new lower bounds for the Shannon capacity of fraction graphs. These constructions extend and recover, in a structured and unified manner, previous lower bounds of Polak and Schrijver \cite[Theorem~9.3.3]{SvenThesis} and Baumert, McEliece, Rodemich, Rumsey, Stanley and Taylor~\cite{MR0337668}. 
    \item \textbf{The Bohman limit via lattices.} Using our Shannon capacity lower bounds, we extend the Bohman limit theorem  \cite{MR1991769, bohman2005limit} from odd cycle graphs to all fraction graphs. In particular, this proves that the group-theoretic (lattice) approach is strong enough to reach the Shannon capacity in the limit of large cycle graphs and fraction graphs.
\end{itemize}

\subsection{Shannon capacity, odd cycles, fraction graphs}
The Shannon capacity of a graph $G$ is defined as $\Theta(G) = \sup_n \sqrt[n]{\alpha(G^{\boxtimes n})}$, where~$\alpha$ denotes the independence number and $\boxtimes$ denotes the strong product of graphs.\footnote{The strong product~$G \boxtimes H$ is the graph with vertex set the cartesian product $V(G) \times V(H)$ and edge set $E(G\boxtimes H) = \{\{(a,x),(b,y)\} : (a=b \vee \{a,b\} \in E(G)) \wedge (x=y \vee \{x,y\} \in E(H)) \}$.} The supremum may equivalently be replaced by a limit (by Fekete's lemma). 

For every $p \in \NN$ let $C_p$ denote the cycle graph with $p$ vertices. For even $p$, it is not hard to see that $\Theta(C_p) = p/2$. Shannon \cite{MR0089131} determined $\Theta(G)$ for all graphs $G$ with at most six vertices, except for the five cycle $C_5$, for which he famously proved $\alpha(C_5^{\boxtimes 2}) \geq 5$ (and thus $\Theta(C_5) \geq \sqrt{5}$) using the independent set $\{t \cdot (1,2) : t \in \ZZ_5\}$ and $\Theta(C_5) \leq 5/2$ via an upper bound method that we now refer to as the fractional clique covering number (and which we denote by $\overline{\chi_f}$). Lovász \cite{Lovasz79} in breakthrough work introduced the Lovász theta function (denoted by $\vartheta$), proved that it upper bounds the Shannon capacity, and used this to determine $\Theta(C_5) = \vartheta(C_5) = \sqrt{5}$. For larger odd cycles, the Shannon capacity is not known. In all cases the Lovász theta function gives the best-known upper bound. For instance, for the seven-cycle we know $ 3.2578\approx 367^{1/5} \leq \Theta(C_7) \leq \vartheta(C_7) = (7\cos(\pi/7))/(1+\cos(\pi/7))\approx 3.3177$ \cite{MR3906144}.

For every $p,q \in \NN$ we let $E_{p/q}$ be the graph with vertex set $\ZZ_p=\ZZ/p\ZZ$ in which two distinct vertices are adjacent if and only if their distance mod $p$ is strictly less than~$q$. We refer to these graphs as the \emph{fraction graphs}. For example, $E_{p/2} = C_p$ is the cycle graph on $p$ vertices, and $E_{p/1} = \overline{K_p}$ is the graph with no edges on $p$ vertices. For any graph $G,H$ we write $G \leq H$ if there is a cohomomorphism $G \to H$ (a map $V(G) \to V(H)$ that maps non-edges to non-edges). The fraction graphs are totally ordered: $E_{a/b} \leq E_{c/d}$ if and only if $a/b \leq c/d$ \cite[Theorem~6.3]{MR2089014}. In particular, equivalent fractions give equivalent fractions graphs in the cohomomorphism order. The fractional clique covering number is given by $\overline{\chi_f}(E_{p/q}) = p/q$ \cite[Cor.~6.20, 6.24]{MR2089014}.

Several upper bounds on the Shannon capacity are known, including the (fractional) clique covering number $\overline{\chi_f}$, the Lovász theta function \cite{Lovasz79} and the (fractional) Haemers bound \cite{haemers1979some, Bukh18}. Theory (asymptotic spectrum duality and distance) has been developed about such upper bounds and their structure \cite{wigderson2022asymptotic, deboer2024asymptotic}. 
In particular, \cite{deboer2024asymptotic} developed a graph limit approach to study Shannon capacity, in which fraction graphs played a central role (to construct converging sequences of graphs). 
It was also observed there that the best-known constructions for lower bounds on Shannon capacity of small odd cycles can all be obtained using ``orbit'' constructions in fraction graphs close to those odd cycle graphs.

\subsection{Bohman limit for fraction graphs}
Bohman \cite{MR1991769, bohman2005limit} proved that, in the limit $p\to\infty$, the Shannon capacity of cycle graphs $\Theta(C_p)$ converges to the fractional clique covering upper bound:
\begin{theorem}\label{eq:intro-1}
$\lim_{p \to \infty} p/2 - \Theta(C_p) = 0.$
\end{theorem}
This was done by constructing a sequence of large independent sets in powers of odd cycles, in \cite{MR1991769} by a direct construction based on an earlier construction of Hales, and in \cite{bohman2005limit} using an expansion process of \cite{MR0337668}.
We strengthen \autoref{eq:intro-1} by proving that the same is true for all fraction graphs:
\begin{theorem}\label{eq:intro-2}
    $\lim_{p/q \to \infty} p/q - \Theta(E_{p/q}) = 0.$
\end{theorem}
Indeed, \autoref{eq:intro-2} implies \autoref{eq:intro-1}, since $E_{p/2} = C_p$.

\subsection{Group-theoretic approach to Shannon capacity}

We strengthen \autoref{eq:intro-2} further by showing that we can obtain the limit via a group-theoretic approach. In other words, the independent sets leading to the limit we will construct as subgroups. 
For any Cayley graph $H$, let $\alphag(H)$ be the size of the largest independent set $S$ of $H$ that is a subgroup, and let $\Thetag(H) = \sup_n \alphag(H^{\boxtimes n})^{1/n}$. We will refer to these quantities as the \emph{subgroup independence number} and the \emph{subgroup Shannon capacity}.
We prove the following: 
\begin{theorem}\label{th:intro-lattice-limit}
For every $\eps > 0$, for all $p/q \in \QQ$ large enough, there are $a,b \in \NN$ such that $a/b = p/q$ and $p/q - \Thetag(E_{a/b}) < \eps$. 
\end{theorem}

\begin{remark}\label{rem:alphag}
In contrast to the independence number~$\alpha$, the subgroup independence number~$\alphag$ can take different values on fraction graphs that are equivalent under cohomomorphism, that is,  $\alphag(E_{\smash{p/q}}^{\boxtimes n})$ is not necessarily equal to $\alphag(E_{\smash{a/b}}^{\boxtimes n})$ if~$a/b=p/q$. For instance, $\alphag(E_{5/2}) = 1$, while $\alphag(E_{10/4}) = 2$. 
Moreover, we note that $\alpha(E_{\smash{p/q}}^{\boxtimes n})$ is not necessarily equal to $\sup_{a,b: a/b\leq p/q} \alphag(E_{\smash{a/b}}^{\boxtimes n})$. Indeed, it is known that $\alpha(E_{\smash{8/3}}^{\boxtimes 3}) = 12$, while $\alphag(E_{\smash{p/q}}^{\boxtimes 3}) < 12$ for all $p/q \leq 8/3$ \cite[Remark~6.12]{deboer2024asymptotic}.
\end{remark}
\begin{remark}
Calderbank et al.~\cite{MR1210400} and Guruswami--Riazanov~\cite{LinearShannonCapacity} proved a barrier result for constructing large independent sets as subspaces, which in our language implies that for $p$ prime we have~$\Theta_{\text{grp}}(E_{p/q}) \leq p^{1-{2(q-1)}/{(p-1)}}$.  For equivalent fractions $a/b=p/q$, this barrier does not necessarily apply to~$E_{a/b}$. It is thus natural to consider all~$a/b$ with~$a/b=p/q$, as we do in our approach (\autoref{th:intro-lattice-limit}). We note that $\sup_{a,b: \, a/b=p/q} \Thetag(E_{a/b})$ is nondecreasing in~$p/q$, as we will see in \autoref{sec:limit}.
\end{remark} 

\subsection{Circle graphs}
Instead of working with the fraction graphs, our results and constructions can naturally be thought of as taking place in infinite graphs on the circle, as follows. 

Let $\mathbb{T} = \mathbb{R}/\mathbb{Z}$ denote the circle group. For $x \geq 2$ let $E_x^{\open}$ be the infinite graph with vertex set $\mathbb{T}$ for which any two distinct vertices are adjacent if and only if their distance modulo $1$ is strictly less than $1/x$. Note that $E_x^{\open}$ is a Cayley graph on $\mathbb{T}$. We prove:

\begin{restatable}{theorem}{circlelimit}
\label{th:intro-circle-limit}
    $\lim_{x \to \infty} x - \Thetag(E_{x}^\open) = 0$.
\end{restatable}

In fact, it can be shown that \autoref{th:intro-circle-limit} is equivalent to \autoref{th:intro-lattice-limit} using that fraction graphs naturally are induced subgraphs of circle graphs.

\subsection{Our construction}

We now describe the bounds that we obtain on the subgroup independence number of powers of fraction graphs from which we will derive \autoref{th:intro-lattice-limit}, and relate them to earlier work. 
\begin{theorem}\label{prop:construction-intro}
    For nonnegative integers~$n$, $k$, $b$, $r$, $s$ with $n,k,b \geq 1$ and $r \leq b$, define
    \[
        a = k\cdot b^n + s \cdot b + r,
        \quad
        p = \displaystyle\frac{s^n r + k a^n}{r + k b^n},
        \quad\text{ and }\quad
        q = \displaystyle\frac{s^{n-1} r + kba^{n-1}}{r + kb^n}.
    \]
Then
    \[
        \alphag(E_{p/q}^{\boxtimes n}) \geq p.
    \]
\end{theorem}

We plot the lower bounds on $\Thetag(E_{p/q})$ that we obtain from \autoref{prop:construction-intro} in \autoref{fig:1}.
Our construction encompasses (and extends) several constructions from the literature, in particular from \cite{MR0337668, SvenThesis, dendulk}.
We note that those previous constructions (even combined with ``expansion methods'' of \cite{MR0337668}) are to our knowledge not sufficient to obtain \autoref{th:intro-lattice-limit}.

\begin{figure}[H]
{\includegraphics[scale=0.515]{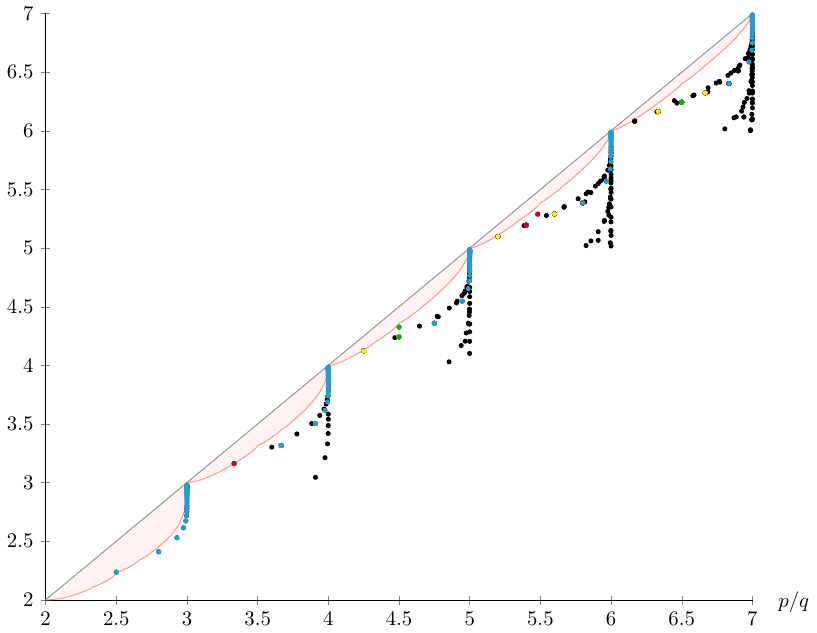}}
\hspace{-5pt}{\includegraphics[scale=0.515]{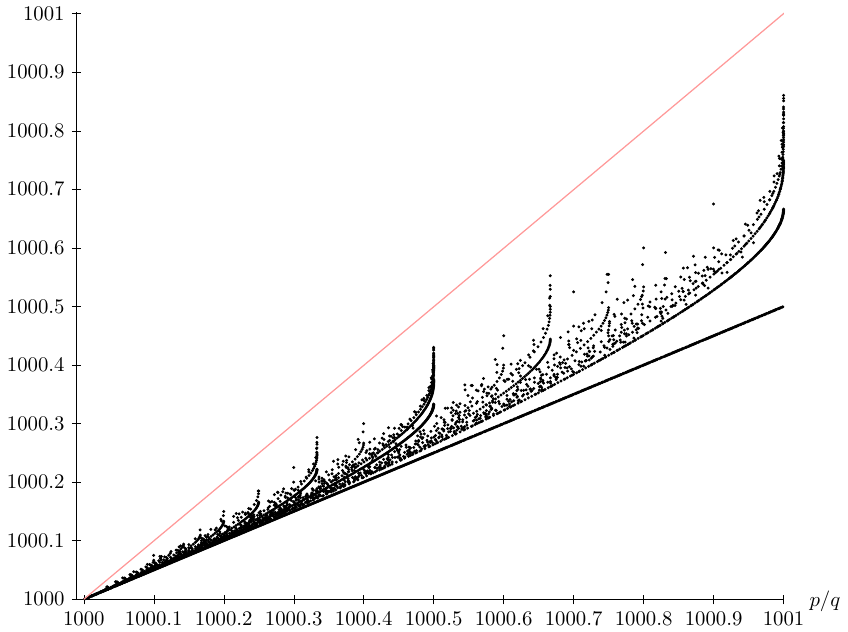}}
\caption{\small Left: Lower bounds on the subgroup Shannon capacity $\Thetag(E_{p/q})$ for $p/q \in [2,7]$ obtained from \autoref{prop:construction-intro}. The red graph is the Lovász theta upper bound. The points indicate our lower bounds. Previously known points of \cite{MR0337668, SvenThesis, dendulk} are indicated with the points coloured purple\,$\color{purple}\bullet$, green\,$\color{green!70!black}\bullet$, blue\,$\color{cyan!80!black}\bullet$, and yellow\,$\color{yellow!80!black}\bullet$,  respectively. Right: Similarly, but for~$p/q\in[1000,1001]$, illustrating behaviour of our lower bounds on $\Thetag(E_{p/q})$ for large $p/q$ (again, a subset of these points were known from aforementioned work, but we do not colour them here). Note that here the Lovász theta function $\vartheta(E_{p/q})$ is very close to (i.e.~essentially indistinguishable from)~$p/q$.}\label{fig:1}
\end{figure}

\paragraph{Comparison to known constructions.} Our family of constructions (\autoref{prop:construction-intro}) encompasses and recovers several known constructions:
\begin{itemize}
    \item[$\color{purple}\bullet$]  \autoref{prop:construction-intro} with $b=2, r=1, s=0$, gives~$\alphag(E_{\smash{p/q}}^{\boxtimes n}) \geq \frac{a^{n}-a^{n-1}}{2^n}$ for~$a=k2^n+1$. These are indicated by purple points in \autoref{fig:1}. Since in this case $p/q =a/2$, this implies the first lower bound in \cite[Theorem 3]{MR0337668}: $\alpha(C_a^{\boxtimes n}) \geq \frac{a^{n}-a^{n-1}}{2^n}$ for~$a=k2^n+1$. %
    \item[$\color{green!70!black}\bullet$]  \autoref{prop:construction-intro} with $b=2, r=1, s=1$, gives~$ \alphag(E_{p/q}^{\boxtimes n}) \geq \frac{a^{n+1}-3a^n+2^n}{2^n(a-2)} =p$ for~$a=k2^n+3$. These are indicated by green points in \autoref{fig:1}. Since in this case $p/q < a/2$, this implies the second lower bound in \cite[Theorem~3]{MR0337668}: $\alpha(C_a^{\boxtimes n}) \geq \frac{a^{n+1}-3a^n+2^n}{2^n(a-2)} =p$ for~$a=k2^n+3$. %
    \item[$\color{cyan!80!black}\bullet$] \autoref{prop:construction-intro} with $b=1, r=1, s=1$, gives a lower bound on $\alphag(E_{\smash{p/q}}^{\boxtimes n})$ that was previously obtained in \cite[Theorem 9.3.3]{SvenThesis}, also via subgroups. These are indicated by blue points in \autoref{fig:1}. These lower bounds were used in~\cite{SvenThesis} to show for any integer $m\geq3$ that, if $s/t \to m$ from below, then~$\Theta(E_{s/t}) \to m$, a phenomenon which can also be recognized in~\autoref{fig:1}.   %
    \item[$\color{yellow!80!black}\bullet$] \autoref{prop:construction-intro} with $s=0$, gives $\alphag(E_{p/q}^{\boxtimes n}) \geq\frac{a^{n}-ra^{n-1}}{b^n}$ for ~$a=kb^n+r$. 
    These are indicated by yellow points in \autoref{fig:1}.
    Since in this case $p/q = a/b$, this implies the lower bound in~\cite[Theorem 3.6]{dendulk}: $\alpha(E_{a/b}^{\boxtimes n})\geq(a^n-ra^{n-1})/b^n$.%
\end{itemize}
We note that the independent sets from~\cite[Theorem 3]{MR0337668} and~\cite[Theorem~3.6]{dendulk} were in general not obtained as subgroups.

\section{Independent sets from lattices}
\label{sec:lattice}

Towards the proof of \autoref{prop:construction-intro}, we discuss in this section the connection between subgroup independent sets in powers of fraction graphs and lattices, and we discuss a method for proving lower bounds on the minimum distance of lattices, which we will use later.

\subsection{Lattices as independent sets in powers of fraction graphs}

Given an invertible matrix $A \in \mathbb{R}^{n \times n}$, we let $\mathcal{L}(A)$ denote the lattice generated by the columns of $A$, that is, $\mathcal{L}(A) = \{A\vect{x} : \vect{x} \in \mathbb{Z}^n\}$.
For $\vect{v} \in \RR^n$ we let $\|\vect{v}\|_\infty = \max_i |\vect{v}_i|$.
Given a lattice $\Lambda$, we define the minimum distance $\lambda_{\infty}(\Lambda) = \min\{\|\vect{v}\|_\infty : \vect{v} \in \Lambda\setminus\{\vect{0}\}\}$.

\begin{lemma}
    \label{lem:matrixtolattice}
    Let $A,B \in \mathbb{Z}^{n \times n}$ and $p,q \in \ZZ_{\geq 1}$. Suppose $AB = p I_n$ and 
    $\lambda_{\infty}(\mathcal{L}(A)) = q$. Then $\alphag(E_{p/q}^{\boxtimes n}) \geq \abs[0]{\det(B)}$.
\end{lemma}

\begin{proof}
    Let $\Lambda = \mathcal{L}(A)$ and $S = \{v \pmod{p} : v \in \Lambda\}$. Note that $S$ is a subgroup of $(\mathbb{Z}/p\mathbb{Z})^n$ and that $p\mathbb{Z}^n$ is a sublattice of $\Lambda$. Let $x \in S$ nonzero and consider the representative $v \in \Lambda$ such that $v_i \in (-p/2,p/2]$ for all indices $i$. Since $x \neq 0$, we know that $v \neq 0$ and thus $|v|_\infty \geq q$. This implies that there exists an index $i$ such that the distance from $x_i$ to $0$ is at least $q$ modulo $p$. It follows that $S$ is an independent set of $E_{p/q}^{\boxtimes n}$. Using \cite[Lemma~10]{dadush}, we have $|S| = \abs[0]{\Lambda / p \ZZ^n} = \abs[0]{\det(p \ZZ^n) / \det(\Lambda)} = p^n / \abs[0]{\det(A)} = \abs[0]{\det(B)}$.
\end{proof}

In \autoref{fig:lattice} we give an illustration of \autoref{lem:matrixtolattice}.

\begin{remark}
A lattice $\Lambda$ such that $p\mathbb{Z}^n \subseteq \Lambda \subseteq \ZZ^n$, like $\mathcal{L}(A)$ in \autoref{lem:matrixtolattice}, is called a $p$-ary lattice \cite[Definition 1]{dadush-lecture9}. 
\end{remark}

\begin{remark}
    \autoref{lem:matrixtolattice} explains how any lattice $\Lambda$ such that $p \mathbb{Z}^n \subseteq \Lambda \subseteq \mathbb{Z}^n$ and $\lambda_\infty(\Lambda) \geq q$ corresponds to a subgroup independent set in $E_{\smash{p/q}}^{\boxtimes n}$. 
    In the other direction, if $S$ is an independent set in  $E_{\smash{p/q}}^{\boxtimes n}$ that is also a subgroup we can consider the lattice $\Lambda_S$ generated by $p \mathbb{Z}^n$ together with a set of representatives of the elements of $S$ in $\mathbb{Z}^n$. It is not hard to see that $S \mapsto \Lambda_S$ gives a one to one correspondence between subgroup independent sets in $E_{\smash{p/q}}^{\boxtimes n}$ and lattices $\Lambda$ such that $p \mathbb{Z}^n \subseteq \Lambda \subseteq \mathbb{Z}^n$ and $\lambda_\infty(\Lambda) \geq q$.
\end{remark}

\subsection{Lower bound on the minimum distance via principal minors}

In general it is a hard problem to determine the minimum distance $\lambda_\infty(\Lambda)$ \cite{vanEmdeBoas1981}. To get a handle on it we will use the concept of a $P_0$-matrix \cite{MR4411340}. For any matrix $M \in \RR^{n \times n}$ and $S \subseteq [n]$ we denote by $M_S$ the principal submatrix $(M_{i,j})_{i,j \in S}$.
We call~$M$ a \emph{$P_0$-matrix} if for every $S\subseteq[n]$ we have $\det(M_S)\geq0$, where we define $\det(M_\emptyset) = 1$. 
\begin{lemma}\label{lem:inf-dist-lb}
    Suppose $M \in \mathbb{R}^{n \times n}$ is a $P_0$-matrix.  
    Then for every $q > 0$  the matrix $A = M + q \cdot I_n$ is invertible and satisfies $\lambda_\infty(\mathcal{L}(A)) \geq q$.
\end{lemma}

We work towards the proof of \autoref{lem:inf-dist-lb}, which we note is closely related to known results about $P_0$-matrices and the stronger notion of $P$-matrices \cite{MR4411340}.

\begin{lemma}\label{lem:det-comp}
    For $M \in \RR^{n \times n}$, $\det(M + \mathrm{diag}(d_1, \ldots, d_n)) = \sum_{S \subseteq [n]} \det(M_S) \prod_{i \not\in S} d_i$.
\end{lemma}
\begin{proof}
    This follows from multilinearity of the determinant and cofactor expansion over appropriate columns.
\end{proof}

\begin{lemma}\label{lem:P0nonneg}
    Let $M \in \RR^{n \times n}$ be a $P_0$-matrix. Then for every $\vect{x} \in \mathbb{R}^n \setminus \{ \vect{0} \}$ there exists an $i \in [n]$ such that  ${\vect{x}}_i \neq 0$ and  $(M \vect{x})_i \, \vect{x}_i \geq 0$.
\end{lemma}
\begin{proof}
    Suppose that there is an $x \in \RR^n \setminus \{0\}$ such that for every $j \in [n]$ we have $x_j = 0$ or  $x_j (Mx)_j < 0$. By restricting the coordinates of $x$ and taking the appropriate submatrix of $M$ we may assume that $x_j\neq 0$ for all $j$. Let $d_j = - (Mx)_j / x_j$. Note that $d_j > 0$. Let $A = M + \mathrm{diag}(d_1, \ldots, d_n)$. Then by \autoref{lem:det-comp}, $\det(A) \geq \prod_{j \in [n]} d_j > 0$. However, by construction of the $d_j$, we have $Ax = 0$, which is a contradiction.
\end{proof}

\begin{proof}[Proof of \autoref{lem:inf-dist-lb}]
Let $\vect{x} \in \ZZ\setminus \{\vect{0}\}$, and let $i \in [n]$ such that $\vect{x}_i\neq0$ and $(M\vect{x})_i \vect{x}_i \geq 0$ (\autoref{lem:P0nonneg}). Then $|( (M+qI_n)\vect{x})_i|= |(M\vect{x})_i + q\vect{x}_i| \geq q |\vect{x}_i| \geq q$. Invertibility of $M + qI_n$ follows from \autoref{lem:det-comp}.
\end{proof}

\begin{figure}[H]
\centering
\includegraphics[scale=0.9]{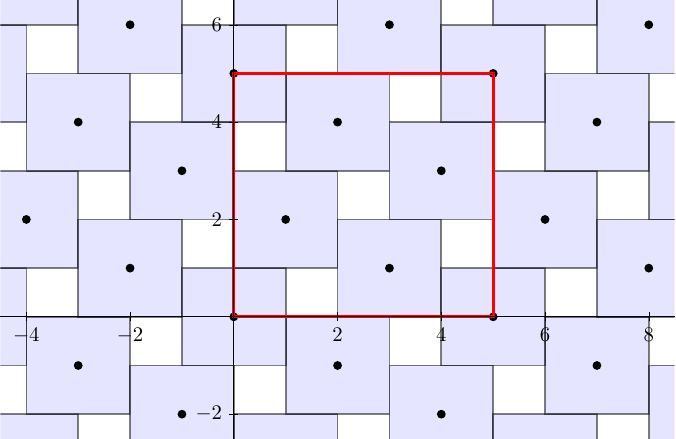}
\caption{The lattice $\Lambda = \mathcal{L}(A) \subseteq \RR^2$ generated by the columns of $A = \left(\begin{smallmatrix} 2 & 1 \\ -1 & 2 \end{smallmatrix}\right)$. Letting $B = \left(\begin{smallmatrix} 2 & -1 \\ 1 & 2 \end{smallmatrix}\right)$, we have that $AB = 5 I_2$, $\abs[0]{\det(B)} = 5$ and $\lambda_\infty(\Lambda) = 2$. It thus follows from \autoref{lem:matrixtolattice} that $\alphag(E_{5/2}^{\boxtimes 2}) \geq 5$. By taking the lattice points inside the red square, one obtains a representation of the independent set $\{\vect{v}\pmod{5} : \vect{v} \in \Lambda\}$. The minimum-distance lower bound $\lambda_\infty(\Lambda) \geq 2$ corresponds in the figure to the fact that the squares of side-length $2$ centered at the lattice points do not overlap, and this lower bound can be proven using \autoref{lem:inf-dist-lb}.}
\label{fig:lattice}
\end{figure}

\section{A family of lattices}
\label{sec:constr_proof}
In \autoref{sec:lattice} we discussed the link between subgroup independent sets in powers of fraction graphs and lattices, and we discussed how to use $P_0$-matrices to lower bound the minimum distance in certain lattices. In this section we will use these ideas to prove \autoref{prop:construction-intro}. First we give an explicit construction of $P_0$-matrices and prove basic properties. Then we use these to construct the lattices that we need.

\subsection{Construction of \texorpdfstring{$P_0$}{P0}-matrices}
The following matrices will play a central role in our construction of lattices, and indeed the main goal here will be to prove that they are $P_0$-matrices.

\begin{definition}
    \label{def:Malpha}
    For $\alpha \in \mathbb{R}$ we define the matrix $M_\alpha \in \mathbb{R}^{n \times n}$ by
\[
(M_{\alpha})_{ij} = \begin{cases} 
      \alpha & \text{if } i < j, \\
      -1 & \text{if } i > j,\\
      0 & \text{otherwise},
   \end{cases}
   \qquad\textnormal{e.g., for $n=4$,}\quad  M_\alpha = \begin{pmatrix}
    0  & \alpha & \alpha & \alpha \\
    -1 & 0 & \alpha & \alpha \\
    -1 & -1 & 0 & \alpha\\
    -1 & -1 & -1 & 0
\end{pmatrix}.
\]
\end{definition}

\begin{lemma}\label{lem:MaP0}\label{lem:M_alpha}%
For $\alpha \in [0,1]$ the matrix $M_\alpha$ is a $P_0$-matrix.
\end{lemma}

Towards the proof of \autoref{lem:MaP0} we first introduce some auxiliary matrices and prove some basic properties.

\begin{definition}
    For $\alpha \in \mathbb{R}$ we define the matrices $R_\alpha, D_\alpha \in \mathbb{R}^{n \times n}$ by
\[
(R_\alpha)_{ij} = \begin{cases} 
      -\alpha & \textnormal{if $i=1, j=n$}, \\
      1 & \textnormal{if $i = j+1$},\\
      0 & \text{otherwise},
   \end{cases}
   \qquad\textnormal{e.g., for $n=4$,}\quad  R_\alpha = \begin{pmatrix}
    0  & 0 & 0 & \hspace{-0.5em}-\alpha \\
    1 & 0 & 0 & 0 \\
    0 & 1 & 0 & 0\\
    0 & 0 & 1 & 0
\end{pmatrix},
\]
and $D_\alpha = \text{diag}(1,\alpha, \dots, \alpha^{n-1})$.
\end{definition}

\begin{lemma}\label{lem:det}
    For $x,y,\alpha \in \mathbb{R}$ we have 
    $\det(x R_\alpha + y I_n) = (-x)^n \alpha + y^n$.
\end{lemma}
\begin{proof}
    Let $A = x R_\alpha + y I_n$. We have $\det(A) = \sum_{\sigma \in S_n} \mathrm{sign}(\sigma) \prod_{i=1}^n A_{i, \sigma(i)}$. There are two summands that are possibly nonzero, namely for $\sigma \in S_n$ equal to the trivial permutation $i\mapsto i$ and for $\sigma$ equal to the cyclic shift $\pi : i \mapsto i-1 \bmod n$. Therefore,
    \[
    \det(A) = \prod_{i=1}^n A_{i,i} + \mathrm{sign}(\pi) \prod_{i=1}^n A_{i, \pi(i)} = y^n + (-1)^{n-1} \cdot x^{n-1} \cdot (-x \alpha).\qedhere
    \]
\end{proof}

\begin{lemma}\label{lem:Ma}
        For $\alpha\neq-1$,
        $M_\alpha = (R_\alpha + \alpha I_n) (R_\alpha - I_n)^{-1}$. %
\end{lemma}

\begin{proof}
    From $\alpha\neq-1$ follows that $R_\alpha - I_n$ is invertible by \autoref{lem:det}.
    We will show that $M_\alpha (R_\alpha - I_n) = R_\alpha + \alpha I_n$. For $k \in [n]$ define the vectors $v_k^{+},v_{k}^- \in \mathbb{R}^{n}$ by $v_k^{+} = e_1 + e_2 + \cdots + e_{k-1}$ and $v_k^{-} = e_{k+1} + e_{k+2} + \cdots + e_{n}$, where $v_1^{+} = v_n^{-} = 0$. With this notation we have $M_\alpha e_k = \alpha v_{k}^+ - v_{k}^-$. Observe that for $k \in [n-1]$ we have 
    \[
        M_\alpha(R_\alpha - I_n)e_k = M_\alpha(e_{k+1} - e_{k}) = \alpha(v_{k+1}^{+} - v_{k}^+) - (v_{k+1}^- - v_{k}^-) = \alpha e_{k} + e_{k+1},
    \]
    which is indeed equal to the $k$th column of $R_{\alpha} + \alpha I_n$. We also have 
    \[
        M_\alpha(R_\alpha - I_n)e_n = M_\alpha(-\alpha e_{1} - e_{n}) = \alpha v_1^{-} - \alpha v_{n}^+ = \alpha (e_{n} - e_1),
    \]
    which is indeed the final column of $R_\alpha + \alpha I_n$.
\end{proof}

\begin{proof}[Proof of \autoref{lem:M_alpha}]
    Using \autoref{lem:Ma},  \autoref{lem:det} and $\alpha \in [0,1]$, it follows that  
    \[
    \det(M_{\alpha}) = \frac{\det(R_{\alpha} + \alpha I_n)}{ \det(R_\alpha - I_n)} = 
    \frac{(-1)^n\alpha+\alpha^n}{(-1)^n\alpha+(-1)^n} = \frac{\alpha+(-\alpha)^n}{\alpha+1} \geq 0.
    \] 
    Since every principal submatrix of $M_\alpha$ is again of the form $M_\alpha$ (but of a smaller size), it follows that~$M_{\alpha}$ is a $P_0$-matrix.
\end{proof}

\begin{lemma}\label{lem:Malphafrac}%
        For $\alpha\neq-1, \beta\neq0$, 
        \[
        x I_n + y D_\beta M_\alpha D_\beta^{-1} = (\beta (x+y) R_{\alpha \beta^{-n}} + (-x + y \alpha)I_n) (\beta R_{\alpha \beta^{-n}} - I_n)^{-1}.
        \]
\end{lemma}
\begin{proof}
    From a direct computation, it follows that $D_\beta R_\alpha D_\beta^{-1} = \beta R_{\alpha \beta^{-n}}$. From this and  \autoref{lem:Ma} we obtain $D_\beta M_\alpha D_\beta^{-1} = (\beta R_{\alpha \beta^{-n}} + \alpha I_n) (\beta R_{\alpha \beta^{-n}} - I_n)^{-1}$. Then the claim follows by writing $I_n = (\beta R_{\alpha \beta^{-n}} - I_n)(\beta R_{\alpha \beta^{-n}} - I_n)^{-1}$ and collecting terms.
\end{proof}

\subsection{Construction of lattices}

In this section we define a family of lattices parameterized by $n,k,b,r,s$. In what follows $n$ will always be a positive integer but, unless stated otherwise, $k,b,r,s$ are allowed to be arbitrary real numbers. We will define numbers $a,p,q$ and matrices $X,Y,A,B$ whose values and entries respectively are given by polynomial expressions with integer coefficients in the parameters $k,b,r,s$ and are thus integers whenever all these parameters are integers. This will not always be clear from the definition, since it will often be more convenient to work with a simplified rational form. When working with such rational expressions, since we know that they represent polynomials, we will omit stipulating that denominators are nonzero when proving identities between them.

\begin{definition}\label{def:pqdef}
    We define
    \[
        a = k b^n + s b + r,
        \quad
        p = \displaystyle\frac{s^n r + k a^n}{r + k b^n},
        \quad\text{ and }\quad
        q = \displaystyle\frac{s^{n-1} r + kba^{n-1}}{r + kb^n}.
    \]
\end{definition}

\begin{lemma}
    \label{lem:pqdef}
    Both $p$ and $q$ are polynomials in the parameters $k,b,r,s$, with integer coefficients, and they satisfy: 
    \begin{enumerate}[label=\upshape(\roman*)]
    \item 
    $p = qs + ka^{n-1}$.\label{eq:num1}
    \item 
    $pb = qa -rs^{n-1}$.\label{eq:num2}
    \item 
    If $b,k > 0$ and $r,s \geq 0$, then $p/q \leq a/b$. \label{eq:num3}
    \end{enumerate}
\end{lemma}
\begin{proof}
Define the sequence $x_0, \dots, x_n$ by $x_0 = 1$, $x_{m+1} = \tfrac ab x_m - \tfrac{rs^m}{b}$. Inductively it follows that
\begin{equation}\label{eq:rec-sol}
x_m = \frac{s^m r + k b^{n-m} a^m}{a-sb}.
\end{equation}
We see that $p = x_n$ and $q = x_{n-1}$.
We substitute $r = a - sb - kb^n$ in \autoref{eq:rec-sol} and rewrite to get
\begin{equation}\label{eq:rec-rw}
x_m = s^m + k b^{n-m} \frac{a^m - (sb)^m}{a-sb}
= s^m + kb^{n-m} \sum_{i=1}^m a^{i-1}(sb)^{m-i}.
\end{equation}
It follows that indeed $p$ and $q$ are polynomials in the parameters $k,b,r,s$. One verifies claim \ref{eq:num1} using \autoref{eq:rec-rw}. Claim~\ref{eq:num2} follows directly from the recursive definition of $(x_m)$ applied to $m = n-1$.

Finally, claim~\ref{eq:num3} follows by observing that under the conditions $0 < bq$. If we divide both sides of claim~\ref{eq:num2} by $bq$ we find $p/q = a/b - rs^{n-1}/bq$, which proves the statement.
\end{proof}

We now define matrices $X,Y \in \mathbb{R}^{n\times n}$ using the matrix $M_\alpha$ defined in \autoref{def:Malpha}.

\begin{definition}
    We define
    \[
    X = \frac{k a^{n-1}}{s} D_{s/a} M_{(rs^n)/(k a^n)} D_{a/s}
    \quad\text{ and }\quad 
    Y = -k b^{n-1} D_{1/b} M_{r/(kb^n)} D_b.
    \]
\end{definition}

One can directly verify that the entries of $X$ and $Y$ are polynomials in the parameters by observing that 
\[
X_{ij} = \begin{cases} 
      r a^{j-i-1} s^{n-(j-i+1)} & \text{if } i < j, \\
      0 & \text{if } i = j, \\
      -k a^{n-(i-j+1)}s^{i-j-1} & \text{if } i > j
   \end{cases}
    \quad
    \text{ and }
    \quad 
    Y_{ij} = \begin{cases} 
      -r b^{j-i-1} & \text{if } i < j, \\
      0 & \text{if } i = j, \\
      k b^{n-(i-j+1)}& \text{if } i > j.
   \end{cases}
\]

\begin{lemma}\label{lem:X-P0}
    Let $n \in \mathbb{N}$, $r,s \in \mathbb{R}_{\geq 0}$, and $k,b \in \mathbb{R}_{\geq 1}$ with $r \leq b$, then $X$ is a $P_0$-matrix.
\end{lemma}

\begin{proof}
    Because the $P_0$ property %
    is a closed condition within the space of $n \times n$ matrices, we may assume that $s$ is nonzero. %
    We claim that $M_{(rs^n)/(k a^n)}$ is a $P_0$-matrix. 
    Then, since the $P_0$ property %
    is preserved under both left and right multiplication with nonnegative diagonal matrices, $X$ is $P_0$. Note that $(rs^n)/(k a^n) \in [0,1]$. Indeed, $rs^n \leq b s^n \leq (bs)^n \leq (k\cdot b^n + s \cdot b + r)^n = a^n \leq ka^n.$
    Then the claim follows from \autoref{lem:M_alpha}.
\end{proof}

We now define the matrices $A,B$. The columns of $A$ represent a basis for the lattice that we are interested in. We subsequently prove that the matrices $A,B$ satisfy the requirements of \autoref{lem:matrixtolattice}.

\begin{definition}
    \label{def:AB}
    We define 
    \[
        A = qI_n + X 
        \quad\text{ and }\quad
        B = \frac{a-r}{b} I_n + Y.
    \]
\end{definition}

\begin{theorem}
    \label{th:main-constr}We have 
    
    \begin{enumerate}[label=\upshape(\roman*)]
        \item\label{th:item:i}   $\lambda_\infty(A) \geq q$, whenever $r,s \in \mathbb{R}_{\geq 0}$, and $k,b \in \mathbb{R}_{\geq 1}$ with $r \leq b$, 
        \item\label{th:item:ii} $AB=pI_n$,
        \item\label{th:item:iii} $\det B=p $.
    \end{enumerate}
\end{theorem}

\begin{proof}
    \ref{th:item:i} Apply \autoref{lem:inf-dist-lb} to $A = qI_n + X$ using that $X$ is a $P_0$ matrix (\autoref{lem:X-P0}).
    
    \ref{th:item:ii} We expand $A$ and $B$ as
    \begin{align*}
    A &= q I_n + \frac{ka^{n-1}}{s} D_{s/a} M_{(rs^n)/(ka^n)} D_{a/s},\\
    B &= \frac{a-r}{b} I_n - k b^{n-1} D_{1/b} M_{r/(kb^n)} D_b.
    \end{align*}
    Applying \autoref{lem:Malphafrac}, and simplifying gives
    \begin{align}
        A &= \bigl((qs + k a^{n-1})R_{r/k} + (r s^{n-1}- qa)I_n\bigr)\cdot (s R_{r/k} - a I_n)^{-1},\nonumber\\
        B &= (s R_{r/k} - a I_n) \cdot (R_{r/k} - b I_n)^{-1}.\label{eq:B}
    \end{align}
    Then
    \[
    AB = \bigl((qs + k a^{n-1})R_{r/k} + (r s^{n-1}- qa)I_n\bigr) \cdot (R_{r/k} - b I_n)^{-1}.
    \]
    By \autoref{lem:pqdef}, $qs + ka^{n-1} = p$ and  $r s^{n-1} - qa = -pb$, so we find $AB = pI_n$.

    \ref{th:item:iii} Using~\autoref{eq:B} and $\det(x R_\alpha + y I_n) = (-x)^n \alpha + y^n$ (\autoref{lem:det}) we find 
    \begin{align*}
        \det(B) &= \frac{\det(s R_{r/k} - a I_n)}{\det(R_{r/k} - b I_n)} 
        = \frac{(-s)^n \tfrac{r}{k} +(-a)^n}{(-1)^n\tfrac{r}{k} + (-b)^n}
        = \frac{s^n \tfrac{r}{k} +a^n}{\tfrac{r}{k} + b^n}
        = p.\qedhere
    \end{align*}

\end{proof}

We can thus combine \autoref{th:main-constr} and \autoref{lem:matrixtolattice} to obtain the following corollary.

\begin{corollary}[\autoref{prop:construction-intro}]\label{cor:construction}
 For nonnegative integers~$n$, $k$, $b$, $r$, $s$ with $n,k,b \geq 1$ and $r \leq b$, and with~$p$ and~$q$ as given in \autoref{def:pqdef}, we have
        $\alphag(E_{p/q}^{\boxtimes n}) \geq p.$
\end{corollary}

\section{Proof of the limit theorem\label{sec:limit}}

In this section we prove \autoref{th:intro-circle-limit}. We will first prove that \autoref{th:intro-circle-limit} is equivalent to \autoref{th:intro-lattice-limit}. We recall that \autoref{th:intro-lattice-limit} directly implies \autoref{eq:intro-2}.

\subsection{Equivalence of 
\autoref{th:intro-lattice-limit} and \autoref{th:intro-circle-limit}}

We recall that there is a cohomomorphism from $E_{p/q}$ to $E_{s/t}$ if and only if $p/q \leq s/t$. In particular,  $p/q \leq s/t$ implies $\alpha(E_{p/q}) \leq \alpha(E_{s/t})$. The same is not true however for $\alphag$. For instance, as noted in \autoref{rem:alphag}, $\alphag(E_{5/2}) = 1$, while $\alphag(E_{10/4}) = 2$. The reason that $\alphag$ is not monotone under the cohomomorphism order is because cohomomorphisms (and indeed graph isomorphisms) need not preserve subgroup structure. Still, when restricted to graphs of the form $E_{p/q}$, there is monotonicity in the following sense. 

\begin{lemma}
    \label{lem:map-subgroup-up}
    Let $a_1, \dots, a_n,b_1, \dots, b_n\in \mathbb{Z}_{\geq 1}$ and let $p,q \in \mathbb{Z}_{\geq 1}$ such that for every $i$, $2 \leq a_i/b_i \leq p/q$. Then there exists $s,t \in \mathbb{Z}_{\geq 1}$ with $s/t = p/q$ such that for each $i$ there is a  group homomorphism from $\mathbb{Z}/a_i\mathbb{Z}$ to $\mathbb{Z}/s\mathbb{Z}$ that is a graph cohomomorphism from $E_{a_i/b_i}$ to $E_{s/t}$.
\end{lemma}

\begin{proof}
    Let $A = a_1 \cdots a_n $.
    It is not hard to check that the map $\mathbb{Z}/a_i\mathbb{Z} \to \mathbb{Z}/Ap\mathbb{Z}$ given by multiplication by $Ap/a_i$ is a cohomomorphism from $E_{a_i/b_i}$ to $E_{Ap/Aq}$. 
\end{proof}

For a given fraction $p/q$ we are interested in the value of $\sup_{a,b: \, a/b=p/q} \Thetag(E_{a/b})$, which we recall gives a lower bound on $\Theta(E_{p/q})$. The previous lemma shows that this value is nondecreasing in $p/q$. It is slightly inconvenient having to consider all equivalent fractions $\{a,b: \, a/b=p/q\}$. To remedy this we will show that we may equivalently consider $\Thetag$ of a well chosen infinite graph.

We recall the definition of the circle graph $E_{x}^{\open}$. This graph is defined as the Cayley graph on $\mathbb{T} = \mathbb{R}/\mathbb{Z}$ where distinct vertices are adjacent if and only if their distance modulo $1$ is strictly less than $1/x$. The following lemma is based on~\cite[Lemma~5.4~(a)]{deboer2024asymptotic}.
\begin{lemma}
    \label{lem:circle-graph-finite-graph-iso}
    Let $x \in \RR_{\geq2}$, $N \in \ZZ_{\geq2}$, and let $S$ be the finite subgroup of $\mathbb{T}$ generated by $1/N$. Then there exists a group isomorphism from $S$ to $\mathbb{Z}/N\mathbb{Z}$ that is a graph isomorphism from $E_x^\open[S]$ to $E_{N/\lceil N/x \rceil}$.
\end{lemma}

\begin{proof}
    Take $v \in S$ let $v' \in \mathbb{R}$ be a representative of $v$ and define $f(v) \in \mathbb{Z}/N\mathbb{Z}$ as the integer $Nv'$ modulo $N$. It is not hard to see that this does not depend on the choice of~$v'$. 

    One verifies directly that the map $f$ is a group isomorphism from $S$ to $\mathbb{Z}/N\mathbb{Z}$. Moreover, distinct $u,v \in S$ are adjacent in $E_x^\open[S]$ if and only if there are representatives $u',v' \in \mathbb{R}$ such that $|u'-v'| < 1/x$, which is the case if and only if $|Nu'-Nv'| < N/x$.
    Since the distance between $Nu'$ and $Nv'$ is an integer, this condition holds if and only if $|Nu'-Nv'| < \lceil N/x \rceil$, which concludes the proof.
\end{proof}

It follows from \autoref{lem:circle-graph-finite-graph-iso} that, for any $p/q$, every $E_{a/b}$ with $a/b = p/q$ is a subgraph of $E_{p/q}^\open$ induced by a subgroup. We can prove the following:

\begin{lemma}
    For $p,q \in \mathbb{Z}_{\geq 1}$ with $p/q \in \mathbb{Q}_{\geq 2}$ we have
    \[
    \Thetag(E^\open_{p/q})= \sup \{ \Thetag(E_{a/b}) : a,b \in \NN,\, a/b=p/q\}.
    \]
\end{lemma}

\begin{proof}
    It follows from \autoref{lem:circle-graph-finite-graph-iso} that any $E_{\smash{a/b}}^{\boxtimes n}$ with $a/b = p/q$ is isomorphic, both as a graph and as a group, to a subgraph of $(E_{p/q}^{\open})^{\boxtimes n}$ induced by a subgroup and thus 
    \[
        \sup \{ \Thetag(E_{a/b}) : a,b \in \NN,\,  a/b=p/q\} \leq \Thetag(E^\open_{p/q}).
    \]
    
    For the other direction, let $I$ be a subgroup independent set in $(E_{\smash{p/q}}^{\open})^{\boxtimes n}$. For $i = 1, \dots, n$ let $I_i = \{x_i : x \in I\} \subseteq \mathbb{T}$ be the $i$th projection of $I$. Then $I_i$ is a subgroup of $\mathbb{T}$ and thus generated by an element $1/N_i$ for some nonnegative integer $N_i$. It follows from \autoref{lem:circle-graph-finite-graph-iso} that $I$ can be mapped to a subgroup independent set of the same size in 
    \[
        E_{N_1 / \lceil q N_1 / p \rceil} \boxtimes \cdots \boxtimes E_{N_n / \lceil q N_n / p \rceil}.
    \]
    Because $N_i / \lceil q N_i / p \rceil \leq p/q$ for all $i$, it now follows from \autoref{lem:map-subgroup-up} that this subgroup independent set can be mapped to a subgroup independent set of the same size in $E_{\smash{s/t}}^{\boxtimes n}$ for some $s/t = p/q$. 
 \end{proof}

It follows that \autoref{th:intro-lattice-limit} and \autoref{th:intro-circle-limit} are equivalent. To prove \autoref{th:intro-circle-limit} we shall repeatedly use the following.

\begin{lemma}
    \label{lem:grp-shannon-nondecreasing}
    The map $\mathbb{R}_{\geq 2} \to \mathbb{R}_{\geq 2}, x \mapsto \Thetag(E^\open_x)$ is nondecreasing.
\end{lemma}
\begin{proof}
    Whenever $x \leq y$ the graph $E_x^\open$ is a spanning subgraph of $E_y^\open$ and thus any independent set of $(E_x^\open)^{\boxtimes n}$ that is also a subgroup remains independent and a subgroup in $(E_y^\open)^{\boxtimes n}$.
\end{proof}

\subsection{Proof of \autoref{th:intro-circle-limit}}

We define the function 
\[
    \Delta(n, k,b,r,s) = \frac{k\cdot b^n + s\cdot b + r}{b} - \left(\frac{\frac{r}{k}s^n + (k\cdot b^n + s\cdot b + r)^n}{\frac{r}{k}+b^n}\right)^{1/n}.
\]

\begin{lemma}
    \label{lem:shan_bound}
    Let $n,k,b,r,s \in \mathbb{Z}_{\geq 0}$ with $1 \leq n, k,b$. Let $a = k\cdot b^n + s\cdot b + r$, then 
    \[
        \frac{a}{b} - \Thetag(E^{\open}_{a/b}) \leq \Delta(n,k,b,r,s).
    \]
\end{lemma}

\begin{proof}
    Let $p,q$ be as in \autoref{def:pqdef}. We use \autoref{cor:construction}, \autoref{lem:grp-shannon-nondecreasing}, and \ref{eq:num3} of \autoref{lem:pqdef} to find 
    \[
        p^{1/n} \leq \alphag(E_{p/q}^{\boxtimes n})^{1/n} \leq \Thetag(E_{p/q}^\open) \leq \Thetag(E_{a/b}^\open),
     \]
    from which the claim follows.
\end{proof}

\begin{lemma}
    \label{lem:Delta_bound}
    For $n,k,b,r,s \in \mathbb{R}_{\geq 0}$ with $1 \leq n,k,b$, and $k,r \leq b$, and $s \leq b^n$ we have:
    \[
        \Delta(n,k,b,r,s) \leq 6b^2/n.
    \]
\end{lemma}

\begin{proof}
    Write $a = k\cdot b^n + s\cdot b + r$ then 
    \[
    \Delta(n,k,b,r,s) = \frac{a}{b} - \left(\frac{\frac{r}{k}s^n + a^n}{\frac{r}{k}+b^n}\right)^{1/n} \leq \frac{a}{b} - \left(\frac{a^n}{\frac{r}{k}+b^n}\right)^{1/n} = \frac{a}{b}\left[1 - \frac{1}{\left(\frac{r}{kb^n} + 1\right)^{1/n}}\right].
    \]
    We calculate 
    \[
    1 - \frac{1}{\left(\frac{r}{kb^n} + 1\right)^{1/n}} = 
    \frac{\left(\frac{r}{kb^n} + 1\right)^{1/n} - 1}{\left(\frac{r}{kb^n} + 1\right)^{1/n}} \leq \left(\frac{k}{r}\right)^{1/n} \cdot b \cdot \left[\left(\frac{r}{kb^n} + 1\right)^{1/n} - 1\right].
    \]
    For any $0 \leq x \leq 1$ we have $(1+x)^{1/n} = e^{\log(1+x)/n} \leq e^{x/n} \leq 1 + 2x/n$. We apply this to $x = \frac{r}{kb^n}$ to obtain:
    \[
        \left(\frac{r}{kb^n} + 1\right)^{1/n} - 1 \leq 2\frac{r}{kb^nn}.
    \]
    Combining the above three inequalities we obtain
    \[
        \Delta(n,k,b,r,s) \leq 2 \cdot \frac{a}{b^n} \cdot \left(\frac{r}{k}\right)^{1-1/n} \cdot \frac{1}{n}.
    \]
    Because $1 \leq k \leq b$, $r \leq b$ and $s \leq b^n$ we have $a = k\cdot b^n + s\cdot b + r \leq 3b^{n+1}$ and $r/k \leq b$. The conclusion follows.
\end{proof}

We now prove the main result of the paper, which we restate for convenience.
\circlelimit*
\begin{proof}
    Let $\varepsilon \in (0,1)$ and let $n \in \mathbb{Z}_{\geq 1}$. Fix an integer $b \geq 2$ with $1/\varepsilon < b \leq 2/\varepsilon$. Now suppose $x \geq b^n$. Let $m = \lfloor x \rfloor$ and note that, because $b^n$ is an integer, we also have $m \geq b^n$. Let $r = \lfloor b(x - m )\rfloor$, then $r \in \{0, 1, \dots, b-1\}$ and $x - 1/b \leq (bm + r)/b \leq x$. Let $n_m$ be the largest integer for which $m \geq b^{n_m}$, observe that $n_m \geq n$. It follows from \autoref{lem:shan_bound} and \autoref{lem:Delta_bound} that
    \[
        \frac{bm + r}{b} - \Thetag(E_{(bm+r)/b}^\open) \leq \Delta(n_m + 1, 1,b,r,m-b^{n_m}) \leq \frac{6 b^2}{n_m+1} \leq \frac{24}{n \cdot \varepsilon^2}.
    \]
    We use the fact that $x \mapsto \Thetag(E_{x}^\open)$ is nondecreasing to conclude 
    \[
        x - \Thetag(E_{x}^\open) \leq \frac{1}{b} + (bm + r)/b - \Thetag(E_{(bm + r)/b}^{\open}) < \varepsilon + \frac{24}{n \cdot \varepsilon^2}.
    \]
    This can be made arbitrarily small which concludes the proof.
\end{proof}

\section{Discussion and open problems}

We have developed a group-theoretic approach to the Shannon capacity, which more specifically for fraction graphs and circle graphs involves constructing $p$-ary lattices with large minimum distance. We provided a large family of constructions using this approach, which in particular we used to show that the group-theoretic approach is optimal in the limit of large fraction graphs. 
We discuss natural open problems and directions for variations and extensions. 

\textbf{Optimality.}
The first is whether the group-theoretic approach is powerful enough to determine the Shannon capacity of all fractions graphs, in the following sense.

\begin{problem}\label{prob:1}
    Is the Shannon capacity $\Theta(E_{p/q})$ equal to the supremum of the subgroup Shannon capacity $\Thetag(E_{a/b})$ over all $a,b \in \NN$ such that $a/b = p/q$? %
\end{problem}

Phrased in terms of the circle graphs, \autoref{prob:1} asks if the Shannon capacity and subgroup Shannon capacity coincide on circle graphs, that is, if $\Theta(E^\open_x) = \Thetag(E^\open_x)$ for every~$x \in \QQ_{\geq2}$.
A related problem is whether subgroup Shannon capacity is strictly increasing:

\begin{problem}
    Is $p/q \mapsto \sup_{a,b \in \NN, a/b=p/q} \Thetag(E_{a/b})$ strictly increasing? Is it strictly increasing at $p/q=2$?
\end{problem}

Indeed, Bohman and Holzman \cite{MR1967195} (strengthened by Zhu \cite{zhu2024improved}) proved that $p/q\mapsto \Theta(E_{p/q})$ is strictly increasing at $p/q = 2$. We do not know whether this is true for $p/q \mapsto \sup_{a,b \in \NN, a/b=p/q} \Thetag(E_{a/b})$. In fact, given the constructions we have currently found, it seems possible that $p/q \mapsto \sup_{a,b \in \NN, a/b=p/q} \Thetag(E_{a/b})$ is constant for $p/q \in [2, 2+\eps]$ for some $\eps > 0$.

    \textbf{Perturbations.} 
    Bohman \cite[Theorem 1.4]{bohman2005limit} described a two-parameter family $(d, \ell)$ of independent sets:
        Let $m = \ell 2^d + 2^{d-1} +1$, $p' = \ell m^{d-1} + m^{d-2}(m-1)/2 $ and $q' = 2 \cdot p' / m$, then 
            $\alpha(E_{p' / q'}^{\boxtimes d}) \geq p'.$
    This family does not directly fit in the family of independent sets described in \autoref{sec:constr_proof}. However, we conjecture that it is \emph{close} to it in the following sense. 

    Let $B_{d, \ell}$ be the matrix described in \autoref{def:AB} for parameters $(n,k,b,r,s) = (d, \ell, 2, 1, 2^{d - 2})$. Let $C$ be the $n \times n$ matrix defined by $C_{i,j} = 1/2$ if $i = j+1$ and $C_{i,j} = 0$ otherwise. Furthermore let $\tilde{C}$ be the matrix $C$ with the nonzero entry on the bottom row set to $0$. Then let $B_{d,\ell}' = B_{d, \ell} + 2^{d-2}\tilde{C}(I_d - C)^{-1}$ and $A_{d,\ell}' = p' \cdot B_{d,\ell}'^{-1}$. We conjecture that $(A_{d,\ell}',B_{d,\ell}')$ satisfy the requirements of \autoref{lem:matrixtolattice} with $\det(B_{d,\ell}') = p'$, which would recover the result of \cite{bohman2005limit}. We moreover conjecture that $\lambda_\infty (\mathcal{L}(A_{d,\ell}')) \geq q'$ can be proven by showing that $A_{d,\ell}' - q' I_d$ is a $P_0$ matrix. We have verified both conjectures for $d \in \{ 3,4,\dots, 10\}$ for all~$\ell$.

    \textbf{Quotients.} 
    Let~$G=E_{p/q}^{\boxtimes n}$. Suppose that~$H$ is an independent set in~$G$ which is simultaneously a subgroup of~$(\ZZ/p\ZZ)^n$. Then we have~$\alpha(G)\geq |H| \cdot \alpha(G/H)$, where~$G/H$ is the quotient graph. By modding out subgroups~$H$ of size~$p$ generated by one element, and searching for optimal independent sets in the quotient graph (with Gurobi), we obtain the following improved bounds. The best previously known bounds are from 
    \cite{jurkiewicz2014some}.
    Unfortunately, it becomes difficult to consider higher product powers as the quotient graphs still grow quickly. To mitigate this, one could explore the effect of modding out subgroups generated by two or more elements, potentially leading to smaller quotient graphs.

    \begin{table}[H]
    \centering
\begin{tabular}{llll}
\toprule
    $G$ & $H$ & $\alpha(G/H)$ & $\leq \alpha(G)$  \\
\midrule
    $E_{14/3}^{\boxtimes 3}$ & $\{ t \cdot (1,2,3) \mid t \in \mathbb{Z}/14\mathbb{Z} \}$ & $6$ & $84^* = 14\cdot 6$ \\
    $E_{17/3}^{\boxtimes 3}$ & $\{ t \cdot (1,2,4) \mid t \in \mathbb{Z}/17\mathbb{Z} \}$ & $9$ & $153 = 17\cdot 9$ \\
    $E_{17/4}^{\boxtimes 3}$ & $\{ t \cdot (1,2,4) \mid t \in \mathbb{Z}/17\mathbb{Z} \}$ & $4$ & $68 = 17\cdot 4$ \\
    $E_{19/3}^{\boxtimes 3}$ & $\{ t \cdot (1,2,3) \mid t \in \mathbb{Z}/19\mathbb{Z} \}$ & $12$ & $228 = 19\cdot 12$ \\
    $E_{19/5}^{\boxtimes 3}$ & $\{ t \cdot (1,2,5) \mid t \in \mathbb{Z}/19\mathbb{Z} \}$ & $2$ & $38 = 19\cdot 2$ \\
\bottomrule
\end{tabular}
\caption{\small New lower bounds on~$\alpha(E_{p/q}^{\boxtimes 3})$ obtained via quotients. The lower bound~$\alpha(E_{14/3}^{\boxtimes 3})\geq 84$ matches the known upper bound and is marked with an asterisk. }
\end{table}

\paragraph{Acknowledgements.} PB was supported by the Dutch Research Council (NWO) grant 639.032.614. JZ was supported by the Dutch Research Council (NWO) Veni grant VI.Veni.212.284. This work used the Dutch national e-infrastructure with the support of the SURF Cooperative using grant no.~EINF-4415 and EINF-8278.

\bibliographystyle{alphaurl}
\bibliography{main.bib}
\addcontentsline{toc}{section}{References}
\end{document}